%% file: main.tex
\documentclass[12pt]{amsart}
\input{header}

\gdef\acl{\operatorname{acl}}
\usepackage{enumerate}

\begin{document}

\title{Sum-product Phenomenon Via Dimension}
\author{Yuyan He \and Sergei Starchenko}

\begin{abstract}
        We show a sum-product phenomenon in fields equipped with abstract dimension theories, which simultaneously generalizes the dimensions in geometric theories and Hrushovski's coarse pseudo-finite dimensions. More precisely, we show that for type-definable sets of positive non-zero dimension, non-expansion in dimension of both the sumset and product set implies the existence of a definable field in the same dimension. Main ingredients of the proof include  dimensional analogues of the Ruzsa triangle inequality and the Pl\"unnecke Ruzsa inequality.
\end{abstract}

\maketitle
\begin{section}{introduction}
  For $A\subseteq F$ subset of a field, we write
    \[A+A=\{a+a':a,\,a'\in A\},\, AA=\{a\cdot a':a,\,a'\in A\}.\] Erd\H{o}s and Szemer\'edi showed one of the first example of the sum-product phenomenon
    \begin{theorem}\cite{Erdős1983}\label{erdosszeme}
        There are constants $C,\delta >0$ such that for any finite set $A$ of integers,
        the following inequality holds
        \[\max\{|A+A|,|AA|\}\geq C|A|^{1+\delta}\]
    \end{theorem}
    
    The above statement fails  for arbitrary integral domain, as in the case  of finite characteristics, taking $A$ to be a finite subfield gives a counterexample. However, failure of the statement can usually be explained by the existance of a finite subfield which is in a sense close to the set $A$. An early result in finite fields is the following theorem by Bourgain, Katz, and Tao
    \begin{theorem}[\cite{BKT}]\label{bkt}
    \begin{enumerate}
        \item For any $\delta\in (0,1)$, there are constants $c(\delta)>0$ and $\epsilon(\delta)>0$ depending only on $\delta$ such that the following holds. Let $p$ be a prime number, and $A\subseteq \Z/p\Z$ such that 
        \[p^\delta\leq |A|\leq p^{1-\delta}.\]
        Then
        \[\max(|A+A|,|AA|)\geq c(\delta)|A|^{1+\epsilon(\delta)}.\]
        \item  For any $\delta\in (0,1)$, there is a constant $C(\delta)>0$ depending only on $\delta$ such that the following holds. Suppose $A\subseteq F$ is a subset of a finite field satisfying $|A|\leq |F|^\delta$, and 
        \[\max(|A+A|,|AA|)\leq K|A|\] for some positive integer $K$. Then, there is a subfield $G\subseteq F$ of size $|G|\leq K^{C(\delta)}|A|$. Moreover, there is $X\subseteq F, |X|\leq K^{C(\delta)}$ and $\xi \in F$ such that $A\subseteq \xi G\cup X$.
    \end{enumerate}
        
    \end{theorem}

One can understand the sum-product phenomenon as saying that small expansion in size of both the sumset and productset indicates the existence of a subfield of comparable size. We prove a result of this form, where the notion of size is given by a model theoretic notion of continuous real dimensions. 

Importantly, we require that the dimension could be encoded in the first order logic of the problem. Then, using the model theoretic concept of generic types with respect to this dimension, we give a simple description of the subfield as a definable set.
\begin{theorem*}[Theorem~\ref{main}]
   Let $F$ be a sufficiently saturated model of a theory expanding the theory of fields, equipped with a continuous real  dimension $\delta$. Let $X$ be a complete type over a small parameter set $A$ satisfying $0<\delta(X+X)=\delta(XX)=\delta(X)<\infty$. Let $p$  be a generic type on $X$ over $A$. Then $\frac{p-p}{p-p}$ is a definable subfield, and $\delta\left(\frac{p-p}{p-p}\right)=\delta(X)$.
\end{theorem*}
The above theorem follows from the following more general statement.
\begin{theorem*}[Theorem~\ref{nonsym-main}]
    Let $F$ be as in the previous theorem. Let $Y$ be a type-definable over $A$ subset satisfying $0<\delta(YY-Y)=\delta(Y)<\infty$. Then $\frac{Y-Y}{Y-Y}$ is a definable field and $\delta(\frac{Y-Y}{Y-Y})=\delta(Y)$.
\end{theorem*}

Theorem \ref{main} is then obtained usiing Lemma~\ref{KT}, which shows that for $p$ is as in Theorem \ref{main}, and $a\models p$, the set $Y=a^{-1}p$ satisfies the assumption of Theorem \ref{nonsym-main}

Models of geometric theories and some structures of finite Morley rank are examples of our setting. However, the more interesting case is when the dimension function takes values in the real numbers as oppose to the integers. A guiding example is Hrushovski's coarse pseudo-finite dimension (see \cite{Hru}), which treats asymptotic bounds in the exponent as a dimension in a suitable ultraproduct. 

Theorem~\ref{main}, applied to Hrushovski's coarse pseudo-finite dimension, recovers the first part of the previously mentioned result of Bourgain, Katz, and Tao in finite fields, and recovers Erd\H{o}s and Szemer\'edi's result in the integers. The precise finitary statement that we get is general across all fields.  

\begin{corollary*}[Corrolary~\ref{maincor}]        
For any $\delta\in (0,1)$, there is $\epsilon(\delta) \in (0,1)$ and $N(\delta)\in \N$ such that for any finite subset $A\subseteq F$ of a field, if $|A+A|,|AA|\leq |A|^{1+\epsilon(\delta)}$ and $|A|\geq N(\delta)$, then there is a subfield $E\subseteq F$ satisfying $|A|^{1-\delta}\leq |E|\leq A^{1+\delta}$. Moreover, we can require that $E = \frac{A'-A'}{A'-A'}$ for some $A'\subseteq A$ with $|A'|\geq |A|^{1-r}$
    \end{corollary*}
    \end{section}
    
\begin{section}{Continuous Real Dimension}
    Let $T$ be a complete first order theory, and let $M\models T$ be a model. By definable sets, we mean definable with parameters unless further specified. A partial type is a consistent collection of formulas, possibly with parameters. We denote by $\operatorname{Def}(M)$ the collection of definable sets defined using parameters from $M$. 
    
    Let $\delta:\operatorname{Def}(M)\rightarrow \R_{\geq 0}\cup\{+\infty\}$ be a function taking values in the non-negative real numbers or infinity. We say that $\delta$ is \emph{a continuous, real-valued dimension} if it satisfies the following axioms:
    \begin{enumerate}[(I)]
    \item For any definable sets $X,Y \in \operatorname{Def}(M)$,
     \begin{enumerate}[(1)]
        \item $\delta(X\cup Y)=\max\{\delta(X),\delta(Y)\}$.
        \item $\delta(X\times Y)=\delta(X)+\delta(Y)$.
        \item $\delta(X)=0$ if $X$ is finite.
        \item $\delta(X\cap Y)\leq \min\{\delta(X),\delta(Y)\}$.
        \item For any definable surjection $f\colon X\to Y$ we have 
        \begin{multline*} \delta(Y)+\inf\{ \delta(f^{-1}(y)\colon y\in Y\} 
        \leq \delta(X) \leq \\\delta(Y)+\sup\{ \delta(f^{-1}(y)\colon y\in Y\}. 
       \end{multline*}
       \end{enumerate}
        \item The function $\delta$ is continuous. I.e., for any formula $\phi(\bar{x},\bar{y})$ and $s<t\in \R$, there is a $\emptyset-$definable $D$ such that 
        \[\{a:\delta(\phi(\bar{x},a))\leq s\}\subseteq D \subseteq \{b:\delta(\phi(\bar{x},b))< t\}. \]
        \end{enumerate}
We will abbreviate and call such a function $\delta$ a dimension, which we always assume to be continuous and real-valued. 

We say that $T$ is a \emph{dimensional theory} if there is a model $M\models T$ and a function $\delta=\delta$ satisfying the above axioms. It is easy to see that, by continuity, $\delta(\phi(\xxx,\aaa))$ is determined by the type of $a$ over $\emptyset$ in the theory $T$. It follows  that $\delta$ extends uniquely to any elementary extensions $N\succeq M$, and passes to elementary equivalent models, and  we simply call $\delta$ the dimension of $T$.

If $T$ expands the theory of fields, we say that it is a theory of dimensional fields, and similarly for expansions of the theory of groups.

A dimension gives rise to associated notions of independence and genericity in the usual way.

\begin{definition}
    Let $M\models T$ be a model of a dimensional theory $T$ with dimension $\delta$. We extend $\delta$ to partial types as follows. Let $\Gamma(\xxx)$ be a partial type over a set of parameters $A\subseteq M$.  We define \[\delta(\Gamma{(\bar x)})=\inf\{\delta(\phi(\xxx)):\Gamma(\xxx)\models \phi(\xxx)\}\] A complete type $p(\bar{x})$ over $A$ extending $\Gamma(\bar{x})$ is a \emph{generic} extension (over A) if $\delta(p(\bar{x}))=\delta(\Gamma(\bar{x}))$. A tuple $\bar{a}\models \Gamma(\bar{x})$ is \emph{generic} in $\Gamma$ over $A$ if it realizes a generic complete type over $A$ extending $\Gamma$.

    For a tuple $\bar{a}$, and a set of parameters $A$, we write $\delta(\bar{a}\,/\,A)$ for $\delta(\tp(\bar{a}\,/\,A))$. Let $\bar{b}$ be another tuple. We say that $\bar{a}$ is \emph{independent} from $\bar{b}$ over $A$ if $\delta(\bar{a}\,/\,\bar{b}A)=\delta(\bar{a}\,/\,A)$, and write $\bar{a}\ind_A \bar{b}$. Note that this is equivalent to saying that $\bar a$ is generic in $\tp(\bar a/A)$ over the larger parameter set $A\bar b$.
\end{definition}

We record more properties of a dimension.

\begin{proposition}
Assume that $M$ is a model of a dimensional theory $T$, with dimension $\delta$.
\begin{enumerate}
     \item(existence) For any tuples $\bar{a},\bar{b}$, and a small set of parameters $C\subset \mathbb{M}$, there is $\bar{a}'\equiv_C \bar{a}$ such that $\bar{a}'\ind _C \bar{b}$.
    \item (additivity) $\delta(\bar{a},\bbb/C)=\delta(\aaa/\bbb C)+\delta(\bbb /C)$.
    \item (symmetry) For $\aaa,\bbb$ such that $\delta(\aaa/C)$, $\delta(\bbb/C)< \infty$, we have $\aaa\ind _C \bbb \iff \bbb \ind_C \aaa$.
    \item Suppose $\bar a \in \acl(\bar b C)$, and $C$ a paramter set. Then $\delta(\bar a/C) \leq \delta(\bar b/C)$ 
    \item (Local Character) For any finite tuple $\bar{a}$, and a set of parameters $C$, there is a subset $C'\subseteq C$ with $|C|\leq \aleph_0$ such that $\bar a\ind _{C'}C$.
    \item (Cauchy Schwarz) For any partial types $X,Y$ and surjective, type-definable function $f:X\rightarrow Y$, there is an inequality
    \[2\delta(X)\leq \delta(Y)+\delta(E),\] where $E =\{(x,x')\in X\times X: f(x)=f(x').\}$
\end{enumerate}
\begin{proof}
    \begin{enumerate}
        \item Follows from axioms 1 and 4.
    \item See \cite[2.10]{Hru}.
    \item Follows from additivity. 
    \item Follows from additivity as well.
    \item As $\delta(\aaa/C)=\min\{\phi(\xxx/C)\in \tp(\aaa/C)\}$, and $\Q$ is dense in $\R$, only countably many $\phi_i(\xxx,c_i)\in \tp(\aaa/C)$, $i\in \omega$ are needed to witness $\delta(\aaa/C)$, and $C'=\{c_i:i\in \omega\}$ works.
    \item We follow the hint in \cite{Hru}. We Work over parameters appearing in the formula defining $f$,  and omit the overline for finite tuples.  Fix $a$ generic in $X$, and let $c = f(a)$. Pick $b\ind_c a$ with $b\equiv_ca$.  
    Then, using additivity, 
    \[\delta(a)+\delta(b)=\delta(f(a))+\delta(f(b))+\delta(a/f(a))+\delta(b/f(b)).\]  As $f(a) = f(b)=c$ and $a\ind_c b$,
    \[\delta(a/f(a))+\delta(b/f(b))=\delta(a,b/c).\] Substitute in $\delta(f(a))=\delta(f(b))=\delta(c)$, we get
    \[\delta(a)+\delta(b) = \delta(c)+\delta(a,b)\leq \delta(Y)+\delta(E).\] 
    \end{enumerate}
\end{proof}
\end{proposition}
\begin{example}
    O-minimal theories and more generaly geometric theories are dimensional, where the dimension is given by $\operatorname{acl}$ dimension. Theory of structures with finite Morley rank is dimensional if Morley rank is definable and satisfies additivity.
\end{example}
\begin{example}[\cite{Hru}]\label{ex:pfdim}
    Let $X_i$ be sets for $i\in \omega$, and $\mathcal{U}$ a non-principal ultrafilter on $\omega$. Define $\mathbf{X} = \prod_{i\rightarrow \mathcal{U}}X_i$ to be the ultraproduct. Pick finite sets $Y_i\subseteq X_i, \, i\in \omega$ such that $\lim_{i\rightarrow \mathcal{U}}|Y_i|=\infty$, and let $\mathbf{Y} = \prod_{i\rightarrow \mathcal{U}}Y_i$. For any $Z_i\subseteq X_i, \, i\in \omega$, define 
    \[\delta_\mathbf{Y} \left(\prod_{i\rightarrow \mathcal{U}}Z_i\right)=\lim_{i\rightarrow \mathcal{U}}\log_{|Y_i|}|Z_i| .\] Under the convention $\log_{a}\infty = \infty$ for finite $a\in \R_{>0}$, the function $\delta_Y$ is defined for all internal sets. 

    Suppose moreover that $X_i's$ are $\mathcal{L}-$ structures for some countable first order language $\mathcal{L}$. Then $\delta_\mathbf{Y}$ satisfies all the axioms for a dimension for the $\mathcal{L}-$ structure $\mathbf{X}$, except continuity. However, continuity can be achieved by adding countably many predicates which name internal sets (see, e.g. \cite{BAYS_2021}).
\end{example}
\end{section}
\begin{section}{Conventions}
     We often work in a sufficiently saturated model of a dimensional theory $T$. We denote by $\delta$ the dimension in the relevant dimensional theory. By definable sets, we mean definable in $\mathcal{L}$, possibly with parameters. Unless stated otherwise, lower case $a,b,c,\dotsc$ denote finite tuples of elements from the model, upper case $A,B,C,\dotsc$ denote small sets of parameters, and $X,Y,Z,\dotsc$ denote definable/type-definable subsets.

      Let $(G,+)$ be an abelian group, and $X,Y$ subsets of $G$. 
    We define the set-wise algebraic operations
    \[X+Y= \{x+y: x\in X,\, y\in Y\},\]
    \[-X=\{-x:x\in X\},\]
    \[nX = \underbrace{X+\cdots + X}_{n \text{ times}}=\{x_1+\cdots +x_n:x_1,\dotsc,x_n\in X\}.\]
    It is important to note that we interpret $nX$ as the $n-$fold power of \emph{set-wise} addition. In the case that $X,Y$ are subsets of a field, we define $XY$, $X^{-1}=\frac{1}{X}$, and $X^n$ for the set-wise multiplication, multiplicative inverse, and multiplicative powers similarly. The symbol $\times$ will be used exclusively for cartesian products. For each polynomial $P(T)\in \Z[T]$, we interpret $P(X)$ using the set-wise operation. And for any rational expression $r(T)=\frac{P(T)}{Q(T)}\in \Z(T)$, 
    \[r(X)=\frac{P(X)}{Q(X)}=\frac{P(X)}{Q(X)\setminus \{0\}}.\]
\end{section}
\begin{section}{Abelian Groups}

 In this section, we prove dimensional analogues of results in additive combinatorics. We assume that $T$ expand the theory of abelian groups.

An important tool that we will use repeatedly is the following version of Ruzsa's triangle inequality. 
  \begin{lemma}[Dimensional Ruzsa Triangle Inequality]\label{triangle}
            Let $X,Y$ and $Z$ be type-definable over $A$ sets in a dimensional group with each of $X,Y$ and $Z$ having finite dimension. Then 
            \[\delta(X)+\delta(Y-Z) \leq \delta(X-Y)+\delta(X-Z).\]
    \end{lemma}            
    \begin{proof} To simplify notations, we assume $A=\emptyset$.
    Notice that the conclusion of the lemma is equivalent to
    \begin{equation}\label{eq:equiv}
        \delta(X\times(Y-Z)) \leq \delta\bigl((X-Y)\times (X-Z)\bigr).
    \end{equation}    
We consider the functions 

        \[F\colon X\times Y \times Z \rightarrow X\times (Y-Z), 
        \text{ given by } (x,y,z) \mapsto (x,y-z);\]
        and \[G\colon X\times Y\times Z \rightarrow (X-Y)\times (X-Z),  \text{ given by } (x,y,z)\mapsto (x-y,x-z).\]

It suffices to show that $\delta(\im(G))\geq \delta(\im(F))$, as $F$ is surjective. 

Pick $b\in Y$ and $c\in Z$ such that $b-c$ is generic in $Y-Z$. Then pick $a\in X$ generic and independent from $b,c$. Notice that $(a,b-c)$ is generic in $X\times (Y-Z)$. We first aim to show that \[\delta(a,b,c/a,b-c)\geq \delta(a,b,c/ a-b,a-c).\] 

By symmetry and monotonicity, $b,c\ind_{b-c} a$, so 
$\delta(b,c/a,b-c)=\delta(b,c/b-c)$. Thus, 
\[ \delta(a,b,c/a,b-c)=\delta(b,c/a,b-c)=\delta(b,c/b-c).\] 
By monotonicity, 
\begin{align*}
\delta(b,c/b-c) &\geq \delta(b,c/a-b,a-c,b-c) \\ 
&=\delta(a,b,c/a-b,a-c,b-c).\end{align*}\
    
Combine the above eqaulity and inequality to get 
\[ \delta(a,b,c/a,b-c) \geq \delta(a,b,c/a-b,a-c,b-c).\]
Since $b-c\in \acl(a-b,a-c)$, we have
\[ \delta(a,b,c/a,b-c) \geq \delta(a,b,c/a-b,a-c).\]
Secondly, as $\delta(a,b,c,a,b-c)=\delta(a,b,c)=\delta(a,b,c,a-b,a-c)$, using additivity, we conclude that \[\delta(a,b-c)= \delta(a,b,c,a,b-c)-\delta(a,b,c/a,b-c)\]\[\leq \delta(a,b,c,a-b,a-c)-\delta(a,b,c/a-c,a-c)=\delta(a-b,a-c)\]
As $(a,b-c)$ is generic in $X\times (Y-Z)$, we have 
$\delta(X\times (Y-Z))=\delta(a,b-c)$, and \eqref{eq:equiv} follows.
\end{proof}

\begin{remark}\label{triangleineq-rmk}
   The followings are two consequence of the triangle inequality that we will use often.

 Let $X,Y,Z$ be type-definable sets in a dimensional abelian group with dimension $\delta$.
    \begin{enumerate}
        \item Assume  $\delta(X+X)=\delta(X)<\infty $. The triangle ineqaulity gives
        \begin{equation*}
            \begin{split}
                \delta(X)+\delta((-Y)-(-X))&\leq \delta(X-(-X))+\delta(X-(-Y))\\
                &= \delta(X+X)+\delta(X+Y).
            \end{split}
        \end{equation*} Canceling $\delta(X)=\delta(X+X)$ on both sides, we have $\delta(X-Y)\leq \delta(X+Y)$. By symmetry, we have $\delta(X-Y)=\delta(X+Y)$.
        \item Assume  $\delta(X+Y)=\delta(X+Z)=\delta(X)<\infty$. Then 
        \begin{equation*}
            \begin{split}
                \delta(X)+\delta(Y-Z)&\leq \delta(X-(-Z))+\delta(X-(-Y))\\
                &\leq 2\delta(X).
            \end{split}
        \end{equation*} We have $\delta(Y-Z)\leq \delta(X)$.
    \end{enumerate}

\end{remark}

\begin{subsection}{Sumset Estimate}
    
In this section, we prove an analogue of Pl\"unnecke-Ruzsa's sumset estimate (Lemma \ref{sumset}), via an  adaption of the strategy of  \cite{petridis}.

As in Pl\"unnecke's original proof, the proof of the result is reduced to proving a lemma about so called commutative layered graphs. We now define the relevant objects.

Let $\Gamma=(U=U_0\sqcup U_1\sqcup U_{2},E)$ be a type-definable graph  in a dimensional theory, meaning that $U$, $U_0,U_1,U_2$ and $E$ are all given by typle-definable sets, with $E$ being an irreflexive and symmetric relation on $U$. The graph $\Gamma$ is \emph{$3-$layered} if for any $x,y\in U$, such that $xEy$, either $x\in U_i,y\in U_{i+1}$, or $x\in U_{i+1}$, $y\in U_{i}$ for $i=0,1$. 

For $a\in U_i$ with $i=0,1$, let $\im(a)=\{y\in U_{i+1}:aEy\}$. Similarly, for $a\in U_i$ with $i=1,2$, define  $\im^{-1}(a)=\{x\in U_{i-1}: xEa\}$. Observe that these sets are type-definable over $a$ and the parameters for $\Gamma$.
We will be considering type-definable 3-layered graphs that are moreover commutative (\cite{Plünnecke1970}).

\begin{definition}
    Let $\Gamma=(U=U_0\sqcup U_1\sqcup U_{2},E)$ be a type-definable layered graph defined over $A$. We say $\Gamma$ is \emph{definably commutative} if there are definable over $A$ families of functions $\{f_{x,y}(z): (x,y)\in U_0\times U_1\}$, and $\{g_{x,y}(z):(x,y)\in U_1\times U_2\}$ such that 
    \begin{itemize}
        \item For $a\in U_0$ and $b\in \im(a)$, $f_{a,b}$ induces an injection from $\im(b)$ into $\im(a)$, with $f_{a,b}(c)\in \im^{-1}(c)$. 
        \item For $a\in U_{1}$ and $b\in \im(a)$, $g_{a,b}$ induces an injection from $\im^{-1}(a)$ into $\im^{-1}(b)$, with $g_{a,b}(c)\in \im(c)$. 
    \end{itemize}
\end{definition}
\begin{example}\label{groupgraph}
     Let $X,Y\subseteq G$ be subsets of an abelian group. Let $\Gamma=\Gamma_{X,Y} = (U_0\sqcup U_1 \sqcup U_2,E)$ denote the layered graph associated with $X$ and $Y$ which we now define. The vertex set of $\Gamma$ is given by the disjoint union of $U_i = X+iY$, $i = 0,1,2$. For $a\in U_i$, $b\in U_{i+1}$, and $i = 0,1$, the edge relation $aEb$ holds if and only if $b = a+c$ for some $c\in Y$.
    
    The commutativity of $\Gamma$ follows from the group being abelian. Indeed, take $a\in U_0$, $b\in \im(a)$, and fix $c\in Y$ such that $b = a+c$. Then $\im(a) = a+Y$ and $\im (b) = b+Y$. Let $f_{a,b}\colon \im (b)\rightarrow \im (a)$ be $f_{a,b}(x) = x-c$. Using commutativity of the group $G$, it is easy to check that $f_{a,b}$ is well-defined and injective. One can define injections $g_{a,b}:\im^{-1}(a)\rightarrow \im^{-1}(b)$ for $a\in U_1$ and $b\in \im (a)$ in a similar way. 
    
    If $G$ is an abelian group definable in a theory $T$ over $A$, then $\Gamma_{X,Y}$ is moreover definably commutative over $A$.
\end{example}
\begin{lemma}\label{layeredgraph1}
Assume that $T$ is a dimensional theory. Let $\Gamma=(p\sqcup V \sqcup V_1,E)
$ be a type-definable 3-layered graph defined over $A$ that is definably commutative, $p$ being a complete type over $A$. Then $\delta(\im(p))\leq \delta(p)+C$ implies that $\delta(\im^2(p)) \leq \delta(p)+2C$ for any $C\in \R_{\geq 0}$.
\end{lemma}

\begin{proof}
       Let $\Gamma$ be as in the lemma. Pick $q$ a generic type in $\im^2(p)$ over $A$. Consider the type-definable set 
       \[V' = \{x\in V:\im^{-1}(x)\cap p \neq \emptyset\text{ and } \im(x)\cap q \neq \emptyset\}.\]
       
       The induced subgraph $\Gamma'$ on $({p\sqcup V' \sqcup q})$ remains a definably commutative graph over $A$, with the same families of definable functions. 
     
       Moreover,
              \[\delta(\im_{\Gamma'}(p))=\delta(\im_{\Gamma}(p)\cap V')\leq \delta(\im_\Gamma(p))\leq \delta(p)+C.\] Without loss of generality, we assume that $\Gamma=\Gamma'$. To simplify notation, we also assume that $A = \emptyset$. Observe that now $V = \im(p)$ and $q = \im(V) = \im^2(p)$.  Write $E = E_1\cup E_2$, where $E_1 = E\restriction ({p\times V})$ and $E_2 = E\restriction (V \times q)$. 

       Define $D = \delta(q)-\delta(V)$. Our goal is to show that $D\leq C$. Indeed, this implies 
       \[\delta(q)-\delta(p)\leq (\delta(q)-\delta(V))+(\delta(V)-\delta(p))\]
       \[\leq (\delta(q)-\delta(\im(p)))+(\delta(\im(p))-\delta(p))\leq D+C\leq 2C.\]
       
       Let $(a,b)\in E_1$ be generic. Note that $a$ is generic in $\im^{-1}(b)$ over $b$, as $\delta(a,b)$ is maximal for $(a,b)\in E_1$. Pick $c\models q$ such that $bE c$.

       By the commutativity of $p\sqcup V \sqcup q$, there is a definable injection $\im^{-1}(b)\hookrightarrow \im^{-1}(c)$, so
       \begin{equation}\label{eq:comm1}
           \delta(a/b)=\delta(\im^{-1}(b))\leq \delta(\im^{-1}(c)).
       \end{equation}  

Now, by \eqref{eq:comm1} and the choice of $D$,
\begin{equation}\label{eq:commgraph1}
    \begin{split}
        \delta(E_1) &= \delta(b)+\delta(a/b) \\
        &\leq \delta(V)+\delta(a/b)\\
        &\leq \delta(q)-D+\delta(\im^{-1}(c))\\
        &\leq \delta(E_2)-D.
    \end{split}
\end{equation}

Let $(b',c')\in E_2 $ be generic, and pick $a'\in \im^{-1}(b')$ generic. By commutativity again, we get
\begin{equation}\label{eq:comm2}
    \delta(c'/\,b')=\delta(\im(b'))\leq \delta(\im (a')).
\end{equation}
Then by choice of $C$ and \eqref{eq:comm2},
\begin{equation}\label{eq:commgraph2}
    \begin{split}
        \delta(E_2) & = \delta(b')+\delta(c'/b')\\
        &\leq \delta(V)+\delta(c/\,b')\\
        &\leq \delta(p)+ C+\delta(\im (a')) \\
        &\leq \delta(E_1)+C.
    \end{split}
\end{equation}

Compare $\eqref{eq:commgraph1}$ and $\eqref{eq:commgraph2}$, we see that $D\leq C$.
\end{proof}

\begin{lemma}[Sumset Estimate]\label{sumset}
            Assume that $T$ is a dimensional theory expanding the theory of abelian groups. Let $X,Y$ be type-definable over $A$ sets such that $\delta(X+Y)=\delta(X)<\infty$. Then $\delta(p+nY)=\delta(X)$ for any generic over $A$ type $p$ on $X$, and $\delta(mY-nY)\leq \delta(X)$, for all $m,n\in \mathbb{N}$.
    \end{lemma}
\begin{proof} Let $X,Y$ be as in the statement of the lemma, and consider the definably commutative graph $\Gamma_{X,Y}=(U_0\sqcup U_1\sqcup U_2,E)$, following the construction in Example \ref{groupgraph}. Let $p$ be a generic type on $U_0$. By Lemma \ref{layeredgraph1}, taking $C=0$, $\delta(p+2Y)=\delta(\im^{2}(p))=\delta(U_0)=\delta(X)$. Now replace the role of $X$ by $p$ and $Y$ by $2Y$. By Lemma \ref{layeredgraph1} again, we have $\delta(p+4Y) = \delta(X)$. 
Inducting on $n$, we get  $\delta(p+2^{n}Y)=\delta(X)$ for all $n\in \N$. Note that for any $m\in \N$, the type-definable sets $p+mY$ and $mY$ are contained in translates of $p+2^NY$ for some large enough $N\in\N$. Hence, $\delta(p+mY) = \delta(X)$ and $\delta(mY)\leq \delta(X)$ for any $m\in \N$. Finally, Remark \ref{triangleineq-rmk} gives $\delta(nY-mY)\leq \delta(nY+mY)\leq \delta(X)$, 
and the lemma is proven. 
\end{proof}
\end{subsection}
\begin{subsection}{Balog-Szemer\'edi-Gowers Lemma}
In order to state the lemma, we first recall the definition of $n-$gons.

\begin{definition}
     A sequence of tuples $a_i\in \mathbb{M}^{x_i}$, $1\leq i \leq n$, is an $n$-gon over a set of parameters $A$, if for any $j\in \{1,\dotsc ,n\}$
    \begin{enumerate}
        
        \item $\delta(a_1,\dotsc a_{j-1},a_{j+1},\dotsc , a_n/A)=\sum_{i\in \{1,\dotsc , n\}, i\neq j}\delta(a_i/A),$
        \item $a_j$ is in $\operatorname{acl}(a_1,\dotsc a_{j-1},a_{j+1},\dotsc , a_n,A)$.
    \end{enumerate}
    In the case of $n=3$ we will simply call it a triangle.
\end{definition}

\begin{example}\label{n-gon}
    Fix $n\geq 3$, and  tuples $a_1,\dotsc , a_n$ with $0<\delta(a_n/A)\leq \delta(a_1/A)=\cdots =\delta(a_{n-1}/A)<\infty$. For $j\in \{1,\dotsc , n\}$, let $a_{\neq j}=(a_1,\dotsc a_{j-1},a_{j+1},\dotsc , a_n)$. Suppose that $a_1,\dotsc a_{n-1}$ are independent over $A$, and for any $j\in \{1,\dotsc , n\}$, the tuple $a_j$ is in $\operatorname{acl}(A,a_{\neq j})$. Then $a_1,\dotsc , a_n$ form an $n-$gon. To see this, it suffices to show that $\delta(a_{\neq j}/A)=(n-1)\delta(a_1/A)$ for any choice of $j\in \{1,\dotsc , n\}$. Note that this will also force $\delta(a_n/A)=\delta(a_1/A)$.
    
    Fix $j\in \{1,\dotsc , n\}$. By additivity and $a_j\in \operatorname{acl}(A,a_{\neq j})$, we have
    \[\delta(a_1,\dotsc , a_n/A)=\delta(a_j/A,a_{\neq j})+\delta(a_{\neq j}/A)=\delta(a_{\neq j}/A).\] Since $a_n\in \operatorname{acl}\{a_1,\dotsc,a_{n-1}\},$ and $a_1,\dotsc , a_{n-1}$ is independent over $A$, we have
    \[\delta(a_1,\dotsc ,a_n/A)=\delta(a_1,\dotsc ,a_{n-1}/A)=(n-1)\delta(a_1/A).\]  Combine the equalities and we have $\delta(a_{\neq j}/A) = (n-1)\delta(a_1/A)$. 
    
    In particular, if $\delta(X)=\delta(Y) =\delta(X-Y)$ for type definable over $A$ sets $X $ and $Y$, then for any $a\ind_A b$ with $a\in X$, $b\in Y$ generic over $A$, the elements $a,\,b,\,a-b$ form a triangle over $A$.
\end{example}

\begin{remark}
\label{ob:ngon}
    Assume that $a_1,\dotsc a_n$ is an $n-$gon over $A$, and $B\supseteq A$. Take $a_1',\dotsc , a_{n-1}'\equiv_{A}a_1,\dotsc , a_{n-1}$, with $ a_1'\dotsc , a_{n-1}'\ind_{A}B$. Pick $a_n'$ with $\tp(a_n'/Aa_1',\dotsc ,a_{n-1}')=\tp(a_n/Aa_1\dotsc a_{n-1})$. Observe that $a_n'\in \operatorname{acl}\{a_1',\dotsc a_{n-1}'\}$, and 
    \[\delta(a_n'/B)\leq \delta(a_n'/A)=\delta(a_n/A)=\delta(a_1/A)=\delta(a_1'/B).\]
    Hence, by Example \ref{n-gon}, $a_1',\cdots, a_n'$ form an $n-$gon over $B$,  and 
    $a_1',\dotsc, a_n' \equiv_A a_1,\dotsc a_n$.
\end{remark}

\begin{lemma}(Weak Balog-Szemer\'edi-Gowers Lemma)\label{6gon}
    Assume that $T$ is a dimensional theory expanding the theory of abelian groups. Let $X$ and $Y$ be type-definable over A subsets with $\delta(X)=\delta(Y)$, and assume that for any generic pair $(a,b)\in X\times Y$ over $A$, we have $\delta(a-b)=\delta(X)$. Then, for any generic $a\in X$, $b\in Y$, there is a $6-$gon $c_0,c_1,c_2, d_0,d_1,d_2$ over $A,a-b$, where $c_i\in X$, $d_i\in Y$ are generic, and $(c_0-d_0)-(c_1-d_1)+(c_2-d_2)=a-b$.
\end{lemma}
This is a weak analogue of a result by Balog and Szemer\'edi \cite{Balog1994}, which was later proved with a different method by Gowers \cite{Gower1998} with improved bounds. Our statement is weak in the sense that we assume all generic pairs $(a,b)\in X\times Y$ satisfy $\delta(a-b)=\delta(X)$. Although the precise bounds in the finitary result does not show up in our infinitary setting, our proof follows the ideas of Gowers, as presented in \cite{TaoVu}.
    \begin{proof}
      Without loss of generality, assume $\delta(X)=1$, and $A = \emptyset$. Pick $\tilde d\in Y$ such that $\tilde d\ind (a,b)$, and $\tilde c\in X$ such that $\tilde c\ind (a,b,\tilde d)$. The assumption about generic pairs in $X\times Y$ says that \[\delta(a-\tilde d)=\delta(\tilde d-\tilde c)=\delta(b-\tilde c) = 1.\] Example~\ref{n-gon} shows that $ (a,\tilde d,a-\tilde d)$, $(\tilde d,\tilde c,\tilde d-\tilde c)$ and $(b,\tilde c,b-\tilde c)$ are triangles over $\emptyset$.

      \medskip
      \noindent\textbf{Claim.} $(a-\tilde d),(\tilde c - \tilde d), (\tilde c-b)$ forms a triangle over $a,b$. 
      \begin{proof}[Proof of the claim.]
         Notice that each of $(a-\tilde d),(\tilde c - \tilde d), (\tilde c-b)$ has dimension at most $1$, by assumptions of the lemma. By the choices of $\tilde c$ and $\tilde d$, we have $\delta((a-\tilde d),(\tilde c-b)/a,b)=\delta(\tilde c, \tilde d/a,b)=2$, so $\delta((a-\tilde d))=\delta(\tilde c-b)=1$. Since$(a-\tilde d)-(\tilde c - \tilde d )+(\tilde c -b)=a-b$, Claim follows from Example~\ref{n-gon}.
      \end{proof}
        
        The next step is to replace pairs $(a,\tilde d)$, $(\tilde c, \tilde d)$ and $(\tilde c,b)$ with pairs that are independent over $a,b$. We do this as follows.
      Pick $c_0\equiv_{a-\tilde d}a$ with $c_0\ind_{(a-\tilde d)} a,b,(\tilde d-\tilde c),(\tilde c-b)$. Since $a\ind (a-\tilde d)$ and $c_0\equiv_{(a-\tilde d)}a$, we have $c_0\ind (a-\tilde d)$, and transitivity gives 
      \[c_0\ind a,b,(a-\tilde d),(\tilde d - \tilde c ),(\tilde c - b).\]
      Pick $d_0\equiv_{(\tilde d-\tilde c)}\tilde d$ with $d_0\ind_{({\tilde d-\tilde c})} a,b,(a-\tilde d),(\tilde d-\tilde c),(\tilde c-b),c_0$. Transitivity shows that $d_0\ind a,b, (a-\tilde d),(\tilde d-\tilde c),(\tilde c-b),c_0$. Finally, pick $c_1\equiv_{(\tilde c-b)} \tilde c$ with $c_1\ind_{(\tilde c-b)} a,b,(a-\tilde d),(\tilde d-\tilde c),(\tilde c-b),c_0,d_0$. Again, by transitivity, $c_1\ind a,b,(a-\tilde d),(\tilde d-\tilde c),(\tilde c-b),c_0,d_0$. Using the fact that $(\tilde d - \tilde c) \in \operatorname{acl}((a-\tilde d), (\tilde c -b))$ and the choices of $c_0,d_0,c_1$, we have

      \begin{equation}
          \begin{split}
              &\delta(c_0,d_0,c_1,(a-\tilde d),(\tilde d-\tilde c),(\tilde c-b)/(a-b))\\&=\delta(c_0,d_0,c_1,(a-\tilde d)(\tilde c-b)/(a-b))\\& \geq \delta(c_0,d_0,c_1,(a-\tilde d),(\tilde c -b)/a,b)\\
              &=5. 
          \end{split}
      \end{equation}
      Since $c_0\equiv_{(a-\tilde d)} a$, we have $c_2:=c_0-(a-\tilde d)\in Y$. Similarly $d_1:=d_0-(\tilde d-\tilde c) \in X$, and $ d_2:=c_1-(\tilde c-b)\in Y$. Then $(c_0,(a-\tilde d), c_2), (d_0,(\tilde d- \tilde c),d_1)$ and $(c_1,(\tilde c - b),d_2)$ are all triangles, by Example \ref{n-gon}.
      
      Hence, by the way we chose them, for $i\in \{1,2,3\}$, $c_i$ and $d_i$ both have dimension $1$ over $a,b$.

      Now, $a-b = c_0- c_2+d_0-d_1+c_1-d_2$, and each $c_i$ or $d_i$ is in the algebraic closure  of the others $c_j,d_j$ over $a,b$. The tuples 
      \[c_0,d_0,c_1,(a-\tilde d),(\tilde d-\tilde c) \text{ and } c_0,c_1,c_2,d_0,d_1,d_2\] are inter-definable over $a,b$, by the choices of $c_0,c_2,d_1,d_2$, so we also have $\delta(c_0,c_1,c_2,d_0,d_1/a,b)=5$. Hence $c_0,c_1,c_2\in X$ and $d_0,d_1,d_2\in Y$ form a $6-$gon over $(a-b)$ by Example \ref{n-gon}.
    \end{proof}
    \end{subsection}
\end{section}

\begin{section}{Katz Tao Lemma}

    In this section, we assume that the theory $T$ expand the theory of fields, and $T$ is dimensional with dimension $\delta$. 

    The goal of this section is to bound the dimension of $pp-pp$ for a generic type $p$ on a type-definable set $X$ with non-zero finite dimension, given small additive and multiplicative doubling of $X$. This is an analogue of Katz and Tao's method \cite[4.2]{KatzTao}.
  
\begin{lemma}\label{(p-q)X3X-2}
    Suppose that $X$ is a type definable over $A$ subset satisfying $0<\delta(X)=\delta(X+X)=\delta(XX
    ) <\infty$. Let $p,q$ be generic types concentrated on $X$ over $A$. Then $\delta((p-q)X^3X^{-2}) = \delta(X)$.

    \begin{proof}
    Without loss of generality, we may assume $\delta(X) = 1$ and $A = \emptyset$. Let $a\in (p-q)X^3X^{-2}$. Choose $a_1\models p$, $a_2\models q$ and $b_1,...,b_5 \in X$, such that $a = (a_1-a_2)\frac{b_1,b_2,b_3}{b_4,b_5}$. Assume that $b_1,\dotsc, b_5\neq 0$, and write $\tilde b = \frac{b_1,b_2,b_3}{b_4,b_5}$. 
    
    By the previous lemma, taking $X=p$ and $Y=q$, there are $c_1,...,c_6$ forming a $6-$gon over $(a_1-a_2)$, such that $a_1-a_2 = (c_1-c_2)-(c_3-c_4) +(c_5-c_6)$, where $c_1,c_3,c_5\models p$ and $c_2,c_4,c_6 \models q$. By Remark \ref{ob:ngon}, we can assume $c_1,\dotsc c_6$ is a $6-$gon over $(a_1-a_2), \tilde b$. 
    The sumset estimate in the multiplicative group gives $\delta(X^4X^{-2})=1$, so $\delta(c_i\tilde b)\leq 1$ for $i=1,\dotsc 6$. Thus
    \[6\geq \delta\left(a,c_1\tilde b,\dotsc, c_6\tilde b\right)\]\[\geq \delta \left( c_1\tilde b,\dotsc, c_6\tilde b/a \right)+\delta(a)\geq 5+\delta(a).\] Hence $\delta(a)\leq 1$. 
    
    So we have that $\delta(a)\leq 1$ for arbitrary $0\neq a\in (p-q)X^3X^{-1}$, hence $\delta((p-q)X^3X^{-2})=1.$
    \end{proof}
\end{lemma}

\begin{theorem}[Dimentional Katz-Tao's Lemma]\label{KT}
    Suppose $X$ is a type-definable over $A$ set with $0<\delta(X)=\delta(XX) = \delta(X+X)<\infty$, and $p,q$ are any generic types on $X$ over $A$. Then $\delta(pq-pq)=\delta(X).$
    \begin{proof}
        Without loss of generality, assume $\delta(X) = 1$ and $A = \emptyset$. Fix $a,a' \models p$ and  $b,b'\models q$. By Lemma \ref{6gon}, applied to the multiplicative group with $Y=X^{-1}$ and Remark \ref{ob:ngon}, we have $ab= \frac{a_1b_1a_3b_3}{a_2b_2}$, where $a_1,a_2,a_3,b_1,b_2,b_3\in X$ form a $6-$gon over $a,b,a'b'$. Then \begin{equation}\label{4-gon}
            \begin{split}
                ab-a'b'&= \frac{(a_1-b')b_1a_3b_3}{a_2b_2}-\frac{b'(a'-b_1)a_3b_3}{a_2b_2}\\
                & +\frac{a'b'(a_3-b_2)b_3}{a_2b_2}-\frac{a'b'b_2(a_2-b_3)}{a_2b_2}.
            \end{split}
        \end{equation}
        \noindent\textbf{Claim:} The elements $\frac{(a_1-b')b_1a_3b_3}{a_2b_2},\frac{b'(a'-b_1)a_3b_3}{a_2b_2},\frac{a'b'(a_3-b_2)b_3}{a_2b_2},\frac{a'b'b_2(a_2-b_3)}{a_2b_2}$ form a $4-$gon over $\{a_2,b_2,a,b,a'b'\}$.
        \begin{proof}
            By Lemma \ref{(p-q)X3X-2}, each term has dimension at most $1$. Each term is in the algebraic closure of $a_2,b_2,a,b,a',b'$ and the other terms by  \eqref{4-gon}. By Example \ref{n-gon}, it suffices to show that \[\delta \left(\frac{b'(a'-b_1)a_3b_3}{a_2b_2},\frac{a'b'(a_3-b_2)b_3}{a_2b_2},\frac{a'b'b_2(a_2-b_3)}{a_2b_2}\Big/ a_2,b_2,a,b,a',b'\right)\geq 3.\] Observe that we can solve for $b_3$, $a_3$ and $b_1$, one by one, given $a_2,b_2,a,b,a',b'$, so $\frac{b'(a'-b_1)a_3b_3}{a_2b_2},\frac{a'b'(a_3-b_2)b_3}{a_2b_2},\frac{a'b'b_2(a_2-b_3)}{a_2b_2}$ are interdefinable with $b_3,a_3,b_1$ over $a_2,b_2,a,b,a',b'$. Also $\delta(b_3,a_3,b_1/a_2,b_2,a,b,a',b')=3$, since $a_1,a_2,a_3,b_1,b_2,b_3$ forms a $6-$gon over $a,b,a',b'$, 
        \end{proof}
        
        Now we compute
        \begin{equation*}
        \begin{split}
            4&\geq\delta \left(\frac{(a_1-b')b_1a_3b_3}{a_2b_2},\frac{b'(a'-b_1)a_3b_3}{a_2b_2},\frac{a'b'(a_3-b_2)b_3}{a_2b_2},\frac{a'b'b_2(a_2-b_3)}{a_2b_2}\right)
            \\&= \delta \left(ab-a'b',\frac{(a_1-b')b_1a_3b_3}{a_2b_2},\frac{b'(a'-b_1)a_3b_3}{a_2b_2},\frac{a'b'(a_3-b_2)b_3}{a_2b_2},\frac{a'b'b_2(a_2-b_3)}{a_2b_2}\right)\\
            &\geq \delta \left(\frac{(a_1-b')b_1a_3b_3}{a_2b_2},\frac{b'(a'-b_1)a_3b_3}{a_2b_2},\frac{a'b'(a_3-b_2)b_3}{a_2b_2},\frac{a'b'b_2(a_2-b_3)}{a_2b_2}/ab-a'b'\right)\\&
            \,+\delta(ab-a'b')\\
            &\geq 3+\delta(ab-a'b'). 
            \end{split}
        \end{equation*}
        Hence $\delta(ab-a'b')\leq 1$.
    \end{proof}
    
\end{theorem}
\end{section}
\begin{section}{Main Theorem}
We first prove the following theorem, whose assumptions are a priori stronger than in the  main theorem. The main theorem then follows easily from Lemma \ref{KT}.
\begin{theorem}\label{nonsym-main}
   Let $F$ be a sufficiently saturated dimensional expansion of a field. Let $Y$ be a type-definable subset satisfying $0<\delta(YY-Y)=\delta(Y)<\infty$. Then $\frac{Y-Y}{Y-Y}$ is a definable field and $\delta(\frac{Y-Y}{Y-Y})=\delta(Y)$.
\end{theorem}

The proof proceeds in two steps. We first identify a property $(\star)$ of elements of the field that form a subfield, an analogue of the basic construction in \cite{tao09}. Then we argue that it is equal to $\frac{Y-Y}{Y-Y}$. The proof only uses the triangle inequality and the sumset estimate.
\begin{definition}
     Let $Y\subseteq F$ be as in the theorem. For an element $x$ of $F$,  we say that $x$ satisfies the property $(\star)$ if
    \[\delta(xY+Y)=\delta(Y).\]
\end{definition}
    
\begin{lemma}\label{star}
    For $Y$ as in Theorem~\ref{nonsym-main},
    the elements with the property $(\star)$ form a subfield.
\end{lemma}
\begin{proof}
        Since $\delta(YY-Y)=\delta(Y)$, any element of $Y$ has $(\star)$, so the set of elements of $F$ with $(\star)$ is not empty. We need to show that it is closed under addition, multiplication and inverses. 
        
        \medskip
\noindent\textbf{Claim A.} Assume $x\in F$ has property $(\star)$, then \[\delta(xY+2Y)=\delta(xY-Y)=\delta(xY+Y-Y)=\delta(Y).\]
\begin{proof}
         As $\delta(YY)=\delta(Y)$, we have $\delta(YY-Y)=\delta(YY)=\delta(Y)$. By the sumset estimate, we have $\delta(nY-mY)=\delta(YY)=\delta(Y)$ for any $n,m\in \N$. In particular, $\delta(3Y)\leq \delta(Y)$. 
         
         Assume $x$ has $(\star)$. By the triangle inequality, \[\delta(Y)+\delta(xY+2Y)=\delta(xY-(-Y))+\delta(Y-(-2Y))=\delta(xY+Y)+\delta(3Y).\] As $\delta(Y)=\delta(xY+Y)$, we have $\delta(xY+2Y)=\delta(3Y)=\delta(Y)$. Notice that Remark \ref{triangleineq-rmk} implies  $\delta(xY-Y)=\delta(Y)$ and $\delta((xY+Y)-Y)=\delta(Y)$. 
\end{proof}
       
    \medskip
    
    \noindent\textbf{Claim B.} The set of elements satisfying $(\star)$ is closed under addition.
    \begin{proof}
     Assume that $x,y$ satisfies $(\star)$, so $\delta(xY+2Y)=\delta(yY+2Y)$ by Claim A. Apply the triangle inequality to get \[\delta(Y)+\delta(xY+Y+yY+Y)\leq \delta(Y-(xY+Y))+\delta(Y-(-(yY+Y))).\] Canceling  $\delta(Y)=\delta(yY+2Y)$ on both side and rearranging, we get
     \[\delta(xY+yY+2Y)\leq \delta(Y-(xY+Y))=\delta(xY+Y-Y)=\delta(Y).\]
     Finally, $(x+y)Y+2Y\subseteq xY+yY+2Y$ implies that $\delta((x+y)Y+2Y)\leq \delta(xY+yY+2Y)\leq \delta(Y)$. Since $(x+y)Y+2Y$ contains additive translates of $(x+y)Y+Y$, we have $\delta((x+y)Y+Y)\leq \delta(Y)$, and $x+y$ has $(\star)$. 
     \end{proof}
     \medskip
     
     \noindent\textbf{Claim C,} The set of elements satisfying $(\star)$ is closed under multiplication.
     \begin{proof}
     Assume that $x$ and $y$ satisfy $(\star)$. Multiplying $yY-Y$ by $x$, we see that $\delta(xyY-xY)=\delta(yY-Y)= \delta(Y)$.  Since, by Claim A,  $x$ satisfies $\delta(xY-Y)=\delta(Y)$, the triangle inequality gives
     \[\delta(xY)+\delta(xyY+Y)\leq \delta(xY-xyY)+\delta(xY-(-Y))=2\delta(Y).\] Cancel $\delta(Y)$ from both sides and we obtain the claim.
    \end{proof}
     \medskip
\noindent\textbf{Claim D.} $(\star)$ is closed under additive and multiplicative inverses.  
\begin{proof}
     For $x$ satisfying $(\star)$, we have $\delta(x^{-1}Y+Y)=\delta(x^{-1}(xY+Y))=\delta(Y)$, and $\delta(-xY+Y)=\delta(xY+Y)=\delta(Y)$, by Remark \ref{triangleineq-rmk}. 
     \end{proof}
\end{proof}

\begin{lemma}
    Let $Y$ be as in Theorem~\ref{nonsym-main}. For any $x\in F$, x satisfies $(\star)$ if and only if $x\in \frac{Y-Y}{Y-Y}$. 
\end{lemma}
    \begin{proof}
        By the assumption $\delta(YY-Y)=\delta(Y)$, any $x\in Y$ satisfies $\delta(xY-Y)\leq \delta(YY-Y)=\delta(Y)$. Using the subset estimate and the triangle inequality
        \[\delta(Y)+\delta(xY+Y)\leq \delta(Y-xY)+\delta(Y-(-Y))=2\delta(Y),\] so $\delta(xY+Y) = \delta(Y)$, and all elements of $Y$ has property $(\star)$. We have shown in the previous Lemma that elements with property $(\star)$ form a field, so in particular all elements of $\frac{Y-Y}{Y-Y}$  satisfy $(\star)$.

        For the other direction, fix $x\in F$ satisfying $(\star)$. Consider the  function 
        \[\phi_x:Y\times Y \rightarrow xY+Y
        \text{ given by }
        (y,y')\mapsto xy+y'.\]
        Clearly, $\phi_x$ is definable over $\{x\}$.
        The property $(\star)$ implies  $\delta(xY+Y)=\delta(Y)<2\delta(Y)=\delta(Y\times Y)$, and, in particular, $\phi_x$ is not injective. Thus, there are $a,b,c,d\in Y$ with $a\neq c$, such that $ax+b=cx+d$. After rearranging, we have $x=\frac{d-b}{a-c}\in \frac{Y-Y}{Y-Y}$.
    \end{proof}
\end{section}
\begin{proof}[Proof of Theorem~\ref{nonsym-main}]
    By the previous two lemmas, $\frac{Y-Y}{Y-Y}$ is a type definable field. By a theorem of Wagner \cite{Wagner}(see Lemma~\ref{fwagner}), it is definable over the same set of parameters. It is left to show that $\delta\left(\frac{Y-Y}{Y-Y}\right)=\delta(Y)$. By the sumset estimate, $\delta(YY-Y)=\delta(Y)$ implies that $\delta(2YY-2YY)=\delta(Y)$. Clearly, $(Y-Y)(Y-Y)\subset 2YY-2YY$. Thus $\delta((Y-Y)(Y-Y))\leq \delta(2YY-2YY)=\delta(Y) $. Applying the sumset estimate in the multiplicative group gives $\delta\left(\frac{Y-Y}{Y-Y}\right)\leq \delta(Y)$.
\end{proof}
Wagner's proof is straightforward, so we include it here.
\begin{lemma}[\cite{Wagner}]\label{fwagner}
     Assume $R\subseteq F$ is a type definable subfield with $0<\delta(R)<\infty$. Then $R$ is definable.
    \begin{proof}
        Let $X\supseteq R$ be a definable set with $\delta(X)< 2\delta(R)$. Choose a definable set  $Y$ with  $R\subseteq Y \subseteq X$, such that $Y+YY\subseteq X$. Fix $a\in Y$, and consider the map $R\times R\rightarrow X$ given by $(x,y)\mapsto x+ay$. This map cannot be injective by the choice of $X$ and $Y$, so $a\in \frac{R-R}{R-R} = R$. I.e. $R\subseteq Y \subseteq R$.  
    \end{proof}
\end{lemma}

\begin{theorem}\label{main}
    Let $F$ be a sufficiently saturated dimensional expansion of a field. Let $X$ be a type-definable over $A$ subset of $F$ satisfying $0<\delta(X+X)=\delta(XX)=\delta(X)<\infty$. Let $p$  be a generic type on $X$ over $A$. Then $\frac{p-p}{p-p}$ is a definable subfield, and $\delta\left(\frac{p-p}{p-p}\right)=\delta(X)$.
\end{theorem}
\begin{proof}
    Fix $a\models p$, and define $Y = a^{-1}p$. By Lemma \ref{KT}, $\delta(pp-pp)=\delta(X)=\delta(Y)$. But $\delta(YY-YY)=\delta(a^2(pp-pp))=\delta(Y)$. As $1\in Y$, $\delta(Y)\leq \delta(YY-Y)\leq \delta(YY-YY)=\delta(Y)$. The rest follows from Theorem \ref{nonsym-main}, noticing that $\frac{Y-Y}{Y-Y}=\frac{p-p}{p-p}$.
\end{proof}
Using the coarse pseudo-finite dimension, we can derive a finitary result.
\begin{corollary}\label{maincor}
    For any $r\in (0,1)$, there is an $s \in (0,1)$ and $N\in \N$ such that for any subset $A\subseteq F$ of a field $F$, if $\max(|A+A|,|AA|)\leq |A|^{1+s}$ and $|A|\geq N$, then there is a subfield $E\subseteq F$ satisfying $|A|^{1-r}\leq |E|\leq A^{1+r}$. Moreover, we can require that $E = \frac{A'-A'}{A'-A'}$ for some $A'\subseteq A$ with $|A'|\geq |A|^{1-r}$
\end{corollary}
\begin{proof}
    Assume not, then there is an $r\in (0,1)$ such that there are no choices of $N$ and $s$ as required. Fix such an $r$. For any positive integer $n$, pick $X_n\subseteq F_n$ with $|X_n|\geq n$ such that $|X_n+X_n|,|X_nX_n|\leq |X_n|^{1+\frac{1}{n}}$, and there is no subfield $E\subseteq F_n$ satisfying $|X_n|^{1-r}\leq |E|\leq |X_n|^{1+r}$ with $E = \frac{A'-A'}{A'-A'}$ for some $A'\subseteq F_n$, $|A'|\geq |X_n|^{1-r}$ .

    Fix a non-principal ultrafilter $\mathcal{U}$ on $\omega$. Let 
    \[(\mathbf{F},\, \mathbf{X} ,\,+,\,\cdot)= \prod_{n\rightarrow \mathcal{U}}(F_n,\, X_n,\, +,\, \cdot).\] 
    As in Example \ref{ex:pfdim}, let 
    \[\delta= \delta_\mathbf{X},\] and equip $(\mathbf{F},\, \mathbf{X} ,\,+,\,\cdot)$ with a first order structure in a countable language $L$ which makes $\delta$ continuous. By the choice of $X_n$, \[\delta(\mathbf{X}) = 1 = \delta(\mathbf{X}+\mathbf{X})=\delta(\mathbf{X}\mathbf{X}).\] Theorem \ref{main} implies that there is a definable $\mathbf{E}\subseteq \mathbf{F}$ such that $\delta(\mathbf{E}) =\delta(\mathbf{X})$. Let $\phi$ be a formula defining $\mathbf{E}$. By \L os's theorem, there are infinitely many $n$ such that $\phi(F_n)$ is a subfield of $F_n$ of size 
    \[|X_n|^{1-r}\leq |E_n|\leq |X_n|^{1+r},\] with $E_n = \frac{A_n-A_n}{A_n-A_n}$, where $A_n\subseteq F_n$ satisfies $|A_n|\geq |X_n|^{1-r}$,
    a contradiction.
\end{proof}
It is easy to see that the above corollary implies Theorem~\ref{erdosszeme}. It also implies the first item of Theorem~\ref{bkt}. 

Indeed, fix $\delta \in (0,1)$. Take $r=\delta$ in Corollary \ref{maincor}. Let $\epsilon=s>0$, and $N$ be given by the corollary. Fix a prime number $q$ with $q^\delta >N$ and take $C(\delta)>0$ be small enough so that $C(\delta)|A|^{1+\epsilon}\leq |A|$ for any $A\subseteq \Z/p\Z$, with $p\leq q$  prime. 

We show that $\epsilon$ and $C=C(\delta)$ work. Let $A\subseteq \Z/p\Z$, for some prime  $p$, and assume  $p^{\delta}\leq|A|$. Suppose $\max(|A+A|,|AA|)< C(\delta)|A|^{1+\epsilon}$. We need to show that $|A|\geq p^{1-\delta}$. By the choice of $C(\delta)$, $p^\delta>N$, and hence $|A|>N$. By the corollary, there is a subfield $E\subset \Z/p\Z$ with $|A|^{1-\delta}\leq|E|\leq |A|^{1+\delta}$. As $|A|> 1$, we must have $E = \Z/p\Z$. Hence $p^{\frac{1}{1+\delta}}\leq |A|$. But $\frac{1}{1+\delta}>1-\delta$, so $|A|\geq p^{1-\delta}$.

\bibliographystyle{plain} \bibliography{text} 
\end{document}

%% file: header.tex
\usepackage{amsmath,amsthm,amssymb,amsfonts, calligra,mathrsfs}
\usepackage[utf8]{inputenc}

\newcommand{\N}{\mathbb{N}}

\newcommand{\Z}{\mathbb{Z}}

\newcommand{\R}{\mathbb{R}}
\newcommand{\Q}{\mathbb{Q}}
\newcommand{\xxx}{\bar{x}}

\newcommand{\aaa}{\bar{a}}
\newcommand{\bbb}{\bar{b}}

\newcommand{\tp}{\operatorname{tp}}
\newcommand{\im}{\operatorname{Im}}
\newtheorem*{theorem*}{Theorem}
\newtheorem*{corollary*}{Corollary}
\newtheorem{theorem}{Theorem}[section]
\newtheorem{corollary}{Corollary}[theorem]
\newtheorem{proposition}{Proposition}[theorem]
\newtheorem{lemma}[theorem]{Lemma}

\theoremstyle{definition}
\newtheorem{example}[theorem]{Example}
\newtheorem{definition}[theorem]{Definition}
\newtheorem{remark}[theorem]{Remark}

\def\Ind#1#2{#1\setbox0=\hbox{$#1x$}\kern\wd0\hbox to 0pt{\hss$#1\mid$\hss}
\lower.9\ht0\hbox to 0pt{\hss$#1\smile$\hss}\kern\wd0}

\def\ind{\mathop{\mathpalette\Ind{}}}